\documentclass{ws-aa}

\usepackage{xcolor}
\usepackage[verbose]{hyperref}
\hypersetup{colorlinks=false,allbordercolors=blue,pdfborderstyle={/S/U/W 1}}

\usepackage{amssymb}
\usepackage{amsmath}
\usepackage[T1]{fontenc}
\usepackage{dsfont}
\usepackage{mathrsfs}
\usepackage[a4paper,asymmetric]{geometry}
\usepackage{mathscinet}
\usepackage{fullpage}
\usepackage{latexsym}
\usepackage{graphicx}
\usepackage{epstopdf}
\usepackage{amssymb}
\usepackage{amsfonts}
\usepackage{algorithm}
\usepackage{algorithmicx}
\usepackage{algpseudocode}
\usepackage{graphicx}
\usepackage{subcaption}

\newcommand{\Ll}{\langle}
\newcommand{\Rr}{\rangle}
 \newcommand{\Z}{\mathbb{Z}}

\newcommand{\tl}{\tilde}
\newcommand{\Be}{\begin{equation}}
\newcommand{\Ees}{\end{equation*}}
\newcommand{\Bes}{\begin{equation*}}
\newcommand{\Ee}{\end{equation}}

\newcommand{\R}{\mathbb{R}}
\newcommand{\E}{\mathbb{E}}

\newcommand{\N}{\mathbb{N}}
\newcommand{\mcl}{\mathcal}

\newcommand{\dif}{\mathrm{d}}

\newcommand{\Bey}{\begin{eqnarray}}
\newcommand{\Eey}{\end{eqnarray}}
\newcommand{\Beys}{\begin{eqnarray*}}
\newcommand{\Eeys}{\end{eqnarray*}}

\begin{document}

%
%

\title{Degenerate Stochastic Delay Modified Equation: Approximation of Stochastic Variance Reduced Gradient}

\author{Ting Zhang\footnote{Corresponding author. Email: zhangtstat@nuist.edu.cn.}}

\address{Nanjing University of Information Science and Technology, Nanjing, China.\\
\email{zhangtstat@nuist.edu.cn.}}

\author{Lulu Tian}

\address{Nanjing University of Information Science and Technology, Nanjing, China.\\
\email{202412381812@nuist.edu.cn}}

\author{Yue Ding}

\address{Hongshan College of Nanjing University of Finance and Economics, Nanjing, China.\\
\email{dingyue.math@163.com}}

\maketitle


\begin{abstract}
According to the property of stochastic variance reduced gradient (SVRG) algorithms, we construct a class of degenerate stochastic delay modified equations (SDMEs). Using the Lindeberg principle and the Markov property, we approximate the SVRG algorithms by the corresponding SDMEs. We obtain order 1 weak approximation in the smooth Wasserstein distance and the theory is validated through numerical experiments.
\end{abstract}



\section{Introduction}\label{first}

~~~~Stochastic gradient descent (SGD) algorithm is a widely used method for massive optimization issues and has been developed into many other stochastic algorithms to solve various types of optimization problems, see  \cite{CDT20,LTE17,Rotskoff18,zhu19},
but its high variance often leads to a slow convergence. To mitigate this, Johnson et al. \cite{J-Z} developed an explicit variance reduction method and proposed the SVRG algorithm.

{SGD and its variants have been widely used in empirical risk minimization. \cite{CL25} investigates three bootstrap-based variants of SGD, which enhance algorithmic stability, statistical robustness, and generalization performance through resampling and aggregation, while also enabling the construction of distribution-free confidence intervals. \cite{NKM24} studies local SGD for overparameterized linear regression and shows that benign overfitting can be achieved in empirical risk minimization. Additional related work includes \cite{FGK15, WZ22, YL21, ZW22}. However, the dynamical behavior of these methods remains insufficiently characterized from a theoretical perspective.} As a generalized formulation of the modified equation method \cite{WH74}, recent works \cite{CSX20,Fan18,feng18,FGL20,FDD20,hu19,LTE17,L-T-E,Simsekli19,Hu2021} have shown that continuous time Stochastic Modified Equation (SME) serve as a weak approximation for SGD algorithms, which provide a powerful framework for analyzing algorithm dynamics. These continuous time models enable insights into convergence behavior and allow the application of tools such as stochastic control theory to design adaptive algorithms \cite{ZYL20,ZXG20}.

In most of the cited works, weak approximations are characterized in some smooth distance, that is,
\Be  \label{e:AppHError}
\sup_{h \in \mcl H} |\E h(z_{n \eta})-\E h(Y_{n\eta})|,
\Ee
where $\mcl H$ is a family of test functions satisfying some certain conditions for differentiability and growth, $z_{n\eta}$ denotes the $n$-th iteration for the stochastic procedure, $\eta$ represents the algorithm's step size, and $Y_{n \eta}$ represents the continuous-time state at step $n \eta$. \cite{L-T-E} constructed a stochastic differential equation with a step size in the noise term as a SME, then derived the error between the SGD algorithm and the SME in a specific smooth distance, and similar distances have been widely used in many fields, see \cite{A-H19,C-X19,JS23,LTE17} and the references therein.

{Extending this line of research, \cite{CLX22} used the non-degenerate Stochastic Delay Differential Equation (SDDE) to approximate the SVRG-Langevin dynamic (LD) algorithms. The perspective of continuous stochastic dynamics offers a powerful approach to understanding SVRG, particularly through continuous time approximations. The most fundamental distinction lies in the mathematical object being approximated. \cite{CLX22} studies SVRG-Langevin dynamics (SVRG-LD), where the inner-loop noise is explicitly governed by an $\eta$-dependent Langevin diffusion. This results in a non-degenerate continuous time model. In contrast, our work studies the classical SVRG algorithm (corresponding to the $\delta \to 0$ limit of SVRG-LD). In this limit, the inner-loop {\it reference gradient} update introduces a structural degeneracy in the associated Stochastic Delay Modified equation (SDME), as shown in \eqref{e:Tr=L2Norm} below. Therefore, the degeneracy is not a minor variant but a fundamental feature arising from the core variance-reduction mechanism of SVRG when divorced from external Langevin noise. By linking SVRG to the SDME model proposed in this paper, we employ SDME theory (\cite{Sch13}) to demonstrate that the delayed diffusion term introduces a dissipative effect, thereby effectively suppressing fluctuations driven by Brownian motion. This provides a conceptual understanding of SVRG's variance-reducing capability.}

Our proof relies on two main techniques. First, we use the Lindeberg principle to control weak approximation errors, yielding a streamlined approach to establishing the normal CLT, a methodology that has since been successfully extended to a wide array of research areas, see \cite{C-S,C-X19,C-D,KoMo11,MePe16,T-V}.
Second, while SGD's iteration sequence forms a Markov chain, enabling specific approximation methods, SDME and SVRG exhibit non-Markovian behavior (due to system memory and internal variance tracking). This fundamental difference renders tools like the Fokker-Planck equation and SGD approximations unsuitable for the latter algorithms.
To overcome this difficulty, we divide the discrete SVRG iterates and the corresponding continuous time SDME process into local patches. Although neither process is Markovian on the whole, patches are individual Markov chains, collectively forming a function-valued Markov chain. Finally, we validate our results through two numerical experiments on optimization problems. These experiments demonstrate that, across varying step sizes and inner loop lengths, the SDME approximation accurately captures the empirical behavior of SVRG, including the norm of the mean and the scaled covariance norm.

\subsection{Notations}

~~~~We employ the Euclidean norm to introduce Lipschitz functions on $\R^d$. Let $g\in\mathcal{C}^{2}(\mathbb{R}^{d})$, the set of second-order continuously differentiable function, then denote the gradient and the Hessian of $g$ by $\nabla g(y)\in \R^d$ and $\nabla^2 g(y)\in \R^{d \times d}$, respectively.

For any $d\times d$ matrices $A,B$, we denote $\Ll A, B\Rr_{{\rm HS}}:=\sum_{i,j=1}^d A_{ij} B_{ij}$ and define the operator norm of the matrix function $\nabla^{2} g(y)$ by
\Beys
\|\nabla^{2} g(y)\|_{\rm op}&=& \sup_{|u|,|v|=1} |\langle\nabla^{2}g(y),uv^{*}\rangle_{{\rm HS}}|
\Eeys
and
\Beys
\|\nabla^{2} g\|_{{\rm op}, \infty}&=& \sup_{y \in \R^{d}} \|\nabla^{2} g(y)\|_{\rm op},
\Eeys
where $*$ is the transpose operator.
Unless specified otherwise, the operator norm subscript "{\rm op}" is frequently omitted in the preceding definitions,
leaving $\|\nabla^{2} g(y)\|=\|\nabla^{2} g(y)\|_{\rm op}$ and  $\|\nabla^{2} g(y)\|_{\infty}=\|\nabla^{2} g(y)\|_{{\rm op}, \infty}$ for brevity.

For a $d\times d$ matrix $B=(B_{ij})_{1\leq i,j\leq d}$, the operator norma can also be defined as $\|B\|_{\rm op} \ = \ \sup_{|v|=1}|Bv|$ and we denote by $\|B\|_{{\rm HS}} \ = \ \sqrt{\sum_{i,j=1}^{d} B^{2}_{ij}} \ = \ \sqrt{{\rm Tr} (B^{*} B)}$, the Hilbert-Schmidt norm. Then, it is easy to verify that
\Be  \label{e:RelEigHS}
 \|B\|_{\rm op} \ \le \ \|B\|_{{\rm HS}} \ \le \ d^{\frac{1}{2}}\|B\|_{\rm op} .
\Ee

The notation $X\stackrel{\rm d}{=}Y$ indicates that the $d$-dimensional random vectors $X$ and $Y$ possess identical distributions. For any $x\geq0$,
denote $\lfloor x\rfloor$ as the greatest integer not exceeding $x$ and $\lceil x\rceil$ represents the largest integer that is strictly less than $x$. The terms  $C_{\cdot}$ refer to positive constants that are unaffected by $\eta,m$,
and their magnitudes might differ in various contexts.

\section{The Mathematical Framework and Main Results}

\subsection{Stochastic Variance Reduced Gradient Algorithms}\label{second}

~~~~Consider a sequence of functions $f_{1},\cdots,f_{n}: \mathbb{R}^{d} \to \mathbb{R}$. A fundamental challenge in optimization is to approximately solve the problem defined by this sequence as follows:
\begin{align}\label{e:1.1}
{ z^*={\rm argmin}_{z\in\mathbb{R}^{d}}F(z)},\qquad F(z):=\frac{1}{n}\sum_{i=1}^{n}f_{i}(z).
\end{align}
For accelerated discovery of the minimizer $z^*$,
\cite{J-Z} developed the SVRG Algorithm \ref{alg1}.

\begin{algorithm}
\caption{ SVRG Algorithm} \label{alg1}
\begin{algorithmic}[1]
\Require
   Number of iterations $N$, an epoch length $m \in \N$, and an step size $\eta>0$
\Ensure  $\tl z_{0}, \tl z_{1}, \cdots$
\medskip
\State \textbf{Initialize} $\tl z_{0} \in \R^{d}$
\Statex \textbf{Outer Loop:}
\For{$k \gets 0$ \textbf{to} $N$}
  \State $z_{k,0}=\tl  z _{k}$
  \State $\tl \mu_{k}=\frac{1}{n}\sum_{i=1}^{n}\nabla f_{i}(\tilde{ z }_{k})$
  \Statex \textbf{Inner Loop (Nested):}
  \For{$t \gets 1$ \textbf{to} $m$}
    \State $\gamma_{t}{\sim} \text{Uniform}(\{1, \cdots, n\})$
    \State $ z _{k,t}= z _{k,t-1}-\eta\left(\nabla f_{\gamma_{t}}( z _{k,t-1})-\nabla f_{\gamma_{t}}(\tilde{ z }_{k})+\tilde{\mu}_k\right)$
  \EndFor
  \State $\tl  z _{k+1}= z _{k,m}$
\EndFor

\State \Return $\tl  z _{0}, \tl  z _{1},\cdots$
\end{algorithmic}
\end{algorithm}

Regarding the aforementioned SVRG algorithm, by defining $k=ml+t$ and setting $ z _{k}= z _{l,t}$, one can readily verify that the algorithmic iteration is expressible as the subsequent difference equation:
\begin{align}\label{representation}
  z _{0}=&\tilde{ z }_{0}, \nonumber \\
  z _{k}=& z _{k-1}-\eta\Big[\nabla f_{\gamma_{k}}( z _{k-1})-\nabla f_{\gamma_{k}}( z _{\lceil\frac{k}{m}\rceil m})+\nabla F( z _{\lceil\frac{k}{m}\rceil m})\Big],\quad k\geq1.
\end{align}
Then, it is evident that
\begin{eqnarray}  \label{d:TlOmega}
\tilde{ z }_s&=& z _{ms}, \quad \quad \ s=0,1,2,\cdots
\end{eqnarray}
To facilitate comparison with an SDME, we express (\ref{representation}) as
\begin{eqnarray}  \label{e:OmeDef}
 z _{k}&=& z _{k-1}-\eta\nabla F( z _{k-1})+\sqrt{\eta}Q_{\eta}( z _{k-1},  z _{\lceil\frac{k}{m}\rceil m}, \gamma_k),
\end{eqnarray}
where
\begin{eqnarray*}
Q_{\eta}( z _{k-1},  z _{\lceil\frac{k}{m}\rceil m}, \gamma_k)
=-\sqrt{\eta}\left[\nabla f_{\gamma_{k}}( z _{k-1})-\nabla f_{\gamma_{k}}( z _{\lceil\frac{k}{m}\rceil m})-\nabla F( z _{k-1})+\nabla F( z _{\lceil\frac{k}{m}\rceil m})\right].
\end{eqnarray*}
Observe that the variable $\gamma_k$ is independently sampled from the set $\{1,2,\cdots,n\}$. Denote by $\gamma$ a uniformly distributed random variable from $\{1,2,\cdots,n\}$. A straightforward calculation yields
\begin{eqnarray}  \label{e:ConV}
\mathbb{E}\big[Q_{\eta}( z _{k-1},  z _{\lceil\frac{k}{m}\rceil m}, \gamma_k)\big| z _{k-1}, z _{\lceil\frac{k}{m}\rceil m}\big]&=&0,
\end{eqnarray}
\begin{eqnarray}  \label{e:ConCov}
{\rm Cov}\big[Q_{\eta}( z _{k-1},  z _{\lceil\frac{k}{m}\rceil m}, \gamma_k)\big| z _{k-1}, z _{\lceil\frac{k}{m}\rceil m}\big]&=&\eta\sigma( z _{k-1}, z _{\lceil\frac{k}{m}\rceil m}),
\end{eqnarray}
where
\begin{eqnarray*}
\sigma\left( z _{k-1}, z _{\lceil\frac{k}{m}\rceil m}\right)
&=&\E_{\gamma} \Big[\big(\nabla f_{\gamma}( z _{k-1})-\nabla f_{\gamma}( z _{\lceil\frac{k}{m}\rceil m})\big) \big(\nabla f_{\gamma}( z _{k-1})-\nabla f_{\gamma}( z _{\lceil\frac{k}{m}\rceil m})\big)^{*} \Big]\\
& \ \ & -\E_{\gamma} \left[\nabla f_{\gamma}( z _{k-1})-\nabla f_{\gamma}( z _{\lceil\frac{k}{m}\rceil m})\right] \left(\E_{\gamma} \left[\nabla f_{\gamma}( z _{k-1})-\nabla f_{\gamma}( z _{\lceil\frac{k}{m}\rceil m})\right]\right)^{*}.
\end{eqnarray*}
Then, if we assume the function $f_{i}(x)$ is $M$-smooth, that is, for $M>0$, $i=1,\dots,n$,
\begin{align}\label{smooth}
\|\nabla f_{i}(x)-\nabla f_{i}(y)\| \leq M\|x-y\|, \qquad \forall x,y\in\mathbb{R}^{d},
\end{align}
we further get that for any $x,y\in\mathbb{R}^{d}$,
\begin{eqnarray}  \label{e:Tr=L2Norm}
\operatorname{tr}\bigl(\sigma( x, y)\bigr)
&=&\mathbb{E}\bigl[\|\nabla f_{\gamma}(x)-\nabla f_{\gamma}(y)\|^{2}\bigr]-\|\nabla F(x)-\nabla F(y)\|^{2}  \nonumber \\
&\le& \mathbb{E}_{\gamma}\bigl[\|\nabla f_{\gamma}(x)-\nabla f_{\gamma}(y)\|^{2}\bigr]\leq M^{2}\|x-y\|^{2}.
\end{eqnarray}
Therefore, it is straightforward to obtain that $\operatorname{tr}\bigl(\sigma( x, x)\bigr)=0$.

{\cite{CLX22} systematically investigates SVRG-Langevin dynamics (SVRG-LD). In that setting, because Langevin dynamics introduces extra noise (depending on the step size $\eta$) within the inner loop, the corresponding continuous-time stochastic equation possesses a non-degenerate diffusion term. Based on this fact, the work employs Malliavin calculus together with a refined Lindeberg principle to derive an error bound of order $(\eta\delta)^{1/2}\ln(1/\eta)$ in the Wasserstein-1 distance between the SVRG-LD algorithm and its stochastic delay differential equation approximation, where $\delta$ denotes the inverse temperature parameter of the Langevin dynamics.

However, \cite{CLX22} imposes the condition $\delta \ge \eta$, which prevents its results from being directly applied to the classical SVRG algorithm (corresponding to the case $\delta=0$). In particular, when $\delta=0$, \eqref{e:Tr=L2Norm} implies that the associated stochastic delay modified equation has a degenerate diffusion term, making the Malliavin calculus technique inapplicable.

To overcome this limitation, the present paper adopts a Jacobian-flow approach (see Appendix \ref{Aind} below). We first establish regularity estimates for the corresponding semigroup and make more thorough use of the structural properties of the processes when comparing one-step errors (see Lemma \ref{SGDtaylor} below). This allows us to rely only on a second-order Taylor expansion (whereas \cite{CLX22} required a third-order expansion), thereby substantially relaxing the regularity requirements on the semigroup and leading to a sharper error bound.}

\subsection{Degenerate Stochastic Delay Modified Equations}

~~~~As evident from \eqref{representation}, the dependence of $ z _{k}$ on both $ z _{k-1}$ and $ z _{\lceil\frac{k}{m}\rceil m}$ implies $\{ z _{k}\}_{k \ge 0}$ is non-Markovian. By contrast, the aggregated process $\{\overline{ z _{k}}\}_{k \ge 0}$, formed by the $m$-tuples $\overline{ z _{i}}=\{ z _{im}, z _{im+1},\cdots, z _{im+m-1}\}$, $i \ge 0$, possesses the Markov property. The structural similarity between this aggregated process and an SDDE provides the impetus for considering the following SDDE as the target modified equation:
\begin{align}\label{sfde}
\dif Y_{t}=-\nabla F(Y_{t})\dif t+\left(\eta\sigma(Y_{t},Y_{\lceil\frac{t}{m\eta}\rceil m\eta})\right)^{\frac{1}{2}}\dif W_{t}, 
\end{align}
here $W_{t}$ represents a standard $d$-dimensional Brownian motion. Hence, we designate \eqref{sfde} as the critical SDME for our investigation. We express that
\Bey  \label{d:TlXs}
\tl Y_s&=&Y_{sm\eta}, \ \ \ \ \ \ s=0,1,2,\cdots
\Eey

One can obtain from \eqref{smooth} that $\nabla F$ is Lipschitz and if we assume the diffusion coefficient matrix $\Sigma(\cdot,\cdot)$ is Lipschitz,
thus the SDME (\ref{sfde}) has a unique solution (see, e.g., \cite{BYY16,Buc00,M}). In addition, we can concluded from \eqref{e:Tr=L2Norm} that the SDME (\ref{sfde}) is degenerate.

\subsection{Main Results}

~~~~Now, in order to compare two stochastic processes, similar to the weak approximation used in \cite{L-T-E}, we consider the following definition of $k$-weak approximation in smooth Wasserstein distance \cite{A-H19}.


Denote $\mathcal{G}^{k}$ as all real-valued functions $\mathbb{R}^{d}\rightarrow\mathbb{R}$ whose partial derivatives exist and are continuous up to order $k$,
belonging to bounded continuous functions. In addition, we denote by $\mathcal{G}^{k}_{1}$ if the upper bounds of its derivatives are all less than or equal to 1.

\begin{definition}\label{metric}
Let $T>0$, $\eta\in(0,1\wedge T)$, and $k\geq 1$ be an integer. Set $N=\lfloor \frac{T}{\eta}\rfloor$. We say that a continuous-time stochastic process $\{Y_{t}\}_{t\in[0,T]}$ is an order $k$ weak approximation of a discrete stochastic process $\{y_{l}\}_{l=0,\cdots,N}$ in smooth Wasserstein distance if for every $g\in\mathcal{G}_{1}^{k+1}$, there exists a positive constant $C$, which is independent of $\eta$, such that
\begin{align*}
\max_{l=0,\cdots,N}\left|\mathbb{E}g(y_{l})-\mathbb{E}g(Y_{l\eta})\right|\leq C\eta^{k}.
\end{align*}
\end{definition}

{Definition \ref{metric} is equivalent to taking $\mathcal{H} = \mathcal{G}_{1}^{k+1}$ in \eqref{e:AppHError}. Consequently, the order-$k$ weak approximation of a discrete stochastic process $\{y_l\}_{l=0,\dots,N}$ in the smooth Wasserstein distance takes the form
\begin{align*}
\max_{0 \le l \le N} \sup_{g \in \mathcal{G}_{1}^{k+1}}\bigl|\mathbb{E} g(y_l) - \mathbb{E} g(Y_{l\eta})\bigr| \le C\eta^{k}.
\end{align*}
This formulation coincides with the smooth Wasserstein distance of order $k+1$ framework widely used in recent literature on weak approximation of limit theorems (see, e.g., \cite{A-H19}).

The distance in \cite{CLX22} is the standard Wasserstein-1 distance, defined via Lipschitz test functions, that is, taking $\mathcal{H}=\mathcal{G}_{1}^{1}$. In Theorem \ref{main} below, we consider the smooth Wasserstein distance of order $2$, which requires the additional constraint $\|\nabla^{2}g\|_{\infty} \leq 1$. This stronger regularity allows us to connect the error directly to the weak error structure $\mathbb{E}[g(\tilde Y_s)]-\mathbb{E}[g(\tl  z _s)]$ via Taylor expansion and the associated generator, which is central to our proof via the Lindeberg principle and semigroup decomposition. The Wasserstein-1 distance does not inherently provide this differentiable path for analysis. In contrast, the class of test functions considered in \cite{L-T-E} is less restrictive in terms of growth conditions: functions there are only required to be twice continuously differentiable, with their first and second derivatives allowed to have polynomial growth, whereas our function class $\mathcal{G}_{1}^{2}$ demands uniformly bounded derivatives. Consequently, the distances defined in both \cite{CLX22} (Wasserstein-1) and \cite{L-T-E} (twice differentiable functions with polynomial growth) are metrically larger than the smooth Wasserstein distance of order $2$ used in Theorem \ref{main} below. The key contribution of our work, however, lies not in the choice of the distance itself, but in systematically applying this type of smooth distance to the degenerate SDME framework corresponding to classical SVRG, and in developing an analytical approach adapted to this degenerate structure.

}

The following theorem, presenting our central result, characterizes the approximation theorem of the stochastic processes $\tl  z _s$ and $\tilde Y_s$.
\begin{theorem}\label{main}
Let $T>0$, $\eta\in(0,1\wedge \frac{T}{m})$, $m\eta\geq1$ and set $N=\lfloor \frac{T}{m\eta}\rfloor$. For any $x,y\in\mathbb{R}^{d}$, let $\nabla P(\cdot),\sigma(\cdot,y)\in\mathcal{G}^{2}$, $\sigma(x,\cdot)\in\mathcal{G}^{1}$, and $ f_{i}(x)$ is $M$-smooth for $i=1,\cdots,n$. Then, {\color{blue}$\{\tilde Y_s\}_{s\in[0,\frac{T}{m}]}$} is an order 1 weak approximation of the SVRG $\{\tl  z _s\}_{s=0,\cdots,N}$, that, is, for each $g\in\mathcal{G}_{1}^{2}$, there exists a constant $C>0$ independent of $m$ and $\eta$ such that
\begin{align*}
\max_{s=0,\cdots,N}\left|\mathbb{E}g({\color{blue}\tilde Y_s})-\mathbb{E}g(\tl  z _s)\right|\leq C\eta.
\end{align*}
\end{theorem}

{
\begin{remark}
The convergence rate $\eta$ obtained in Theorem \ref{main} holds uniformly over the test function class $\mathcal{G}_{1}^{2}$. This uniformity is achieved through explicit control on the derivative growth of the semigroup associated with the SDME, yielding a bound of the following form:
\begin{align*}
\max_{s=0,\dots,N}\sup_{g\in\mathcal{G}_{1}^{2}}\Bigl|\mathbb{E}g(\tilde{Y}_s)-\mathbb{E}g(\tilde{z}_s)\Bigr|\leq C\eta,
\end{align*}
where the constant $C$ depends solely on model parameters (coefficients, dimension, etc.) and is independent of the individual test function $g$.
\end{remark}
}

\section{The Proof of Theorem \ref{main}}\label{forth}

~~~~From \cite[Lemma 3.1]{CLX22}, one can obtain that both external processes $(\tilde  z _s)_{s \in \Z^{+}}$ and $(\tilde Y_s)_{s \in \Z^{+}}$ are Markov processes.
Consequently, the approach to prove Theorem \ref{main} entails two components.
Subsection \ref{internal sub} employs a Lindeberg principle \cite{CSX20,L-T-E} to bound the approximation error for the internal Markov processes
$\{ z _{k}\}_{ms\leq k\leq m(s+1)}$ and $\{Y_{t}\}_{ms\eta\leq t\leq m(s+1)\eta}$.
Subsection \ref{external sub} employs this Lindeberg principle exclusively for approximating the external Markov process $\{\tilde{ z }_{s}\}_{s\geq0}$ with $\{\tilde{Y}_{s\eta}\}_{s\geq0}$.

\subsection{Internal Markov Model Approximation}\label{internal sub}
\medskip

~~~~While the sequence $\{ z _k\}_{0 \le k \le m}$ generated by an inner loop of SVRG is a Markov process on $\R^d$,
the SDME \eqref{sfde} restricted to $[0, m\eta]$  displays as
\begin{align}\label{sfde-0}
\dif Y_{t}=-\nabla F(X_{t})\dif t+\left(\eta\sigma(Y_{t},Y_0)\right)^{\frac{1}{2}}\dif W_{t} \ \ \ {\rm for} \ t \in [0,m\eta].
\end{align}
Under the condition that the initial value is fixed at $Y_{0}= z _{0}$, the preceding SDME coincides with the SDE:
\begin{align}\label{sfde-0-0}
\dif Y_{t}=-\nabla F(Y_{t})\dif t+\left(\eta\sigma(Y_{t}, z _0)\right)^{\frac{1}{2}}\dif W_{t} \ \ \ {\rm for} \ t \in [0,m\eta],
\end{align}
then the solution to equation \eqref{sfde-0-0} $\{Y_{t}\}_{t\in[0,m\eta]}$ is also a Markov process. Moreover, both $\{ z _k\}_{0 \le k \le m}$ and $\{Y_{t}\}_{t\in[0,m\eta]}$ exhibit time-homogeneous Markov properties. For simplicity, for any $0\leq s\leq t\leq \eta$ and $y\in\mathbb{R}^{d}$, we write $Y_{s,t}^y$ to emphasize the Markov process's dependence on the value $Y_s=y$.
Leveraging the time-homogeneous property, we rewrite the notation $Y_{s,t}^y$ as $Y_{t-s}^y$.
The term $ z ^{y}_{k}$ is denoted accordingly.

Note that the coefficients of SDE (\ref{sfde-0-0}) are twice order differentiable,
the estimates below are stated here and proved in \ref{Aind}.

\begin{lemma}\label{mainlem1}
Let $Y_{t}$ be defined by (\ref{sfde-0-0}) and for any $g\in\mathcal{G}_{1}^{2},$ denote $P_{t}g(y)=\mathbb{E}[g(Y_{t}^{y})]$. Then, for arbitrary $y\in\mathbb{R}^{d}$, whenever $\eta\in(0,1]$ and $t\in[0,m\eta]$, the following holds:
\begin{align}\label{gradient2}
|\nabla P_{t}g(y)|\leq e^{Cm\eta},
\end{align}
\begin{align}\label{sec}
\|\nabla^{2}P_{t}g(y)\|_{{\rm HS}}\leq C_{1}e^{C_{2}m\eta},
\end{align}
where $C,C_{1},C_{2}$ are positive constants, and are independent of $m$ and $\eta$.
\end{lemma}
Next, we present moment estimates for the SDME and SVRG methods, their proofs appear in \ref{moment estimate}.

\begin{lemma}\label{fourth}
Let $Y_{t}$ be defined by (\ref{sfde-0-0}) and $\eta\in(0,1)$. Then it follows for any $y\in\mathbb{R}^{d}$ and $t\in[0,m\eta]$,
\begin{align}\label{moment}
\mathbb{E}|Y_{t}^{y}|^{2}\leq C_{1}(1+|y|^{2}+\mathbb{E}|z_{0}|^{2})e^{C_{2}m\eta}
\end{align}
and
\begin{eqnarray}\label{moment-1}
\mathbb{E}|Y_{t}^{y}-y|^{2}
&\le&C_{1}(1+|y|^{2}+\mathbb{E}|z_{0}|^{2})e^{C_{2}m\eta}t(t+\eta),
\end{eqnarray}
where $C_{1}$ and $C_{2}$ are some positive constants not depending on $m$ or $\eta$.
\end{lemma}

\begin{lemma}\label{coupling2}
Define $ z _{k}^{y}$ as in (\ref{representation})  and let $\eta\in(0,1)$. Then, for any $0\leq k\leq m$, we have
\begin{align}\label{SGDforth}
\mathbb{E}| z _{k}^{y}|^{2}\leq C_{1}(1+|y|^{2}+\mathbb{E}|z_{0}|^{2})e^{C_{2}m\eta},
\end{align}
where $C_{1},C_{2}$ are certain positive constants that do not depend on $\eta$ or $m$.
\end{lemma}

Additionally, we require the following lemma, with proof appearing in \ref{moment estimate}.

\begin{lemma}\label{SGDtaylor}
For $0\leq t\leq m\eta,$ denote $v_{t}(y)=\E g(Y^y_{t})$, thus for $\eta\in(0,1)$, we have
\begin{align*}
\left|\mathbb{E}v_{t}(Y_{\eta}^{y})-\mathbb{E}v_{t}(z_{1}^{y})\right|
\leq&C_{1}e^{C_{2}m\eta}(1+|y|^{2}+\mathbb{E}|z_{0}|^{2})\eta^{2},
\end{align*}
where $C_{1},C_{2}$ are certain positive constants that do not depend on $m$ or $\eta$.
\end{lemma}

\begin{proposition}\label{regularity}
For any $x,y\in\mathbb{R}^{d}$, let $\nabla P(\cdot),\sigma(\cdot,y)\in\mathcal{G}^{2}$ and $\sigma(x,\cdot)\in\mathcal{G}^{1}$. For all $i=1,\cdots,n$, assume the function $f_{i}(x)$ is $M$-smooth. { For any $\eta\in(0,1),$ $0\leq k\leq m$ and $g\in\mathcal{G}_{1}^{2}$}, we have
\begin{align*}
\left|\mathbb{E}g(Y_{k \eta})-\mathbb{E}g(z_k))\right|
 \leq C_{1}e^{C_{2}m\eta}(1+\mathbb{E}|z_{0}|^{2})m\eta^{2},
\end{align*}
for certain positive constants $C_{1}$ and $C_{2}$, which are independent of $m$ and $\eta$.
\end{proposition}

\begin{proof}
 {When $k=0$, the conclusion is obvious. For $k\geq1$, denote} $Y_{0}=z _{0}$ and $v_{t}(y)=\mathbb{E}[g(Y_{t}^{y})]$ for $g\in\mathcal{G}_{1}^{2}$ and $t\in[0,m\eta]$. Now, we will use the Lindeberg principle and Markov property to prove the proposition.
 For simplicity in subsequent usage, given any $x\in\mathbb{R}^{d}$ and integers $0\leq r\leq t$,
$Y_{\eta t}(\eta r,x)$ denotes the random variable $Y_{\eta t}$ conditioned on $Y_{\eta r}=x$, and $z_{t}(r,z)$ is defined analogously.
 Then we have
\begin{align*}
\mathbb{E}g(Y_{\eta k})=\mathbb{E}g\left(Y_{\eta k}(\eta,Y_{\eta})\right)-\mathbb{E}g\left(Y_{\eta k}(\eta,z_{1})\right)+\mathbb{E}g\left(Y_{\eta k}(\eta,z_{1})\right).
\end{align*}
When $k>1$, since $Y_{\eta t}=Y_{\eta t}(\eta r,Y_{\eta r})$ for any integers $0\leq r\leq t$, we then have
\begin{align*}
\mathbb{E}g\left(Y_{\eta k}(\eta,z_{1})\right)=\mathbb{E}g\left(Y_{\eta k}\left(2\eta,Y_{2\eta}(\eta,z_{1})\right)\right)-\mathbb{E}g\left(Y_{\eta k}(2\eta,z_{2})\right)+\mathbb{E}g\left(Y_{\eta k}(2\eta,z_{2})\right).
\end{align*}
Continuing iteratively by applying the condition $Y_{\eta t}=Y_{\eta t}(\eta r,Y_{\eta r})$ for all integers $r,t>0$ with $t\geq r$, it immediately follows
\begin{align*}
\mathbb{E}g(Y_{\eta k})=\sum_{j=1}^{k}\left[\mathbb{E}g\left(Y_{\eta k}(j\eta,Y_{j\eta}((j-1)\eta,z_{j-1}))\right)-\mathbb{E}g\left(Y_{\eta k}(j\eta,z_{j})\right)\right]+\mathbb{E}g(z_{k}),
\end{align*}
then
\begin{align*}
\mathbb{E}g(Y_{\eta k})-\mathbb{E}g(z_{k})=\sum_{j=1}^{k}\left[\mathbb{E}g\left(Y_{\eta k}(j\eta,Y_{\eta j}((j-1)\eta,z_{j-1}))\right)-\mathbb{E}g\left(Y_{\eta k}(j\eta,z_{j})\right)\right].
\end{align*}
Notice that $Y_{t}$ is a time homogeneous Markov chain, then it follows $v_{\eta(k-j)}(y)=\mathbb{E}\left[g(Y_{\eta k})|Y_{\eta j}=y\right].$ Hence by the fact that $z_{t}=z_{t}(r, z _{r})$ for any integers $r,t>0$ with $t\geq r$, we have
\begin{align}\label{Linde}
\mathbb{E}g(Y_{\eta k})-\mathbb{E}g(z_{k})=&\sum_{j=1}^{k}\left[\mathbb{E}v_{\eta(k-j)}\left(Y_{\eta j}((j-1)\eta,z_{j-1})\right)-\mathbb{E}v_{\eta(k-j)}(z_{j})\right]\nonumber\\
=&\sum_{j=1}^{k}\left[\mathbb{E}v_{\eta(k-j)}\left(Y_{\eta j}((j-1)\eta,z_{j-1})\right)-\mathbb{E}v_{\eta(k-j)}\left(z_{j}(j-1,z_{j-1})\right)\right]\nonumber\\
=&\sum_{j=1}^{k}\left[\mathbb{E}v_{\eta(k-j)}\left(Y_{\eta}^{z_{j-1}}\right)-\mathbb{E}v_{\eta(k-j)}\left(z_{1}^{z_{j-1}}\right)\right],
\end{align}
here the last inequality is obtained by $Y_{\eta}^{z_{j-1}}\stackrel{\rm d}{=}Y_{\eta j}((j-1)\eta,z_{j-1})$ and $z_{1}^{z_{j-1}}\stackrel{\rm d}{=} z_{j}(j-1,z_{j-1})$. Therefore, by Lemmas \ref{SGDtaylor} and \ref{coupling2}, we have
\begin{align*}
\left|\mathbb{E}g(Y_{\eta k})-\mathbb{E}g(z_{k})\right|
\leq&\sum_{j=1}^{k}\left|\mathbb{E}v_{\eta(k-j)}\left(Y_{\eta}^{z_{j-1}}\right)-\mathbb{E}v_{\eta(k-j)}\left(z_{1}^{z_{j-1}}\right)\right|\\
\leq&C_{1}e^{C_{2}m\eta}\eta^{2}\sum_{j=1}^{k-1}(1+\mathbb{E}|z_{j-1}|^{2})(1+\mathbb{E}|z_{0}|^{2})\\
\leq& C_{1}e^{C_{2}m\eta}(1+\mathbb{E}|z_{0}|^{2})k\eta^{2}\leq C_{1}e^{C_{2}m\eta}(1+\mathbb{E}|z_{0}|^{2})m\eta^{2},
\end{align*}
where $C_{1},C_{2}$ are certain positive constants that do not depend on $m$ or $\eta$. The proposition is proved.
\end{proof}

\subsection{External Markov Model Approximation}\label{external sub}

~~~~Recall $\tl Y_s$ in \eqref{d:TlXs}. Let $g \in\mathcal{G}_{1}^{2}$, define
\Bey  \label{e:Uh}
V_g(s,y)&=&\E \left[g(\tl Y^y_s)\right], \ \ \ \ s=0,1,\cdots,N,
\Eey
where the superscript {`$y$' in $\tl Y^y_s$} denotes the initial value.

The following lemma shows that the function $V_g(s,\cdot)\in\mathcal{G}^{2}$ and the proof is postponed to \ref{moment estimate}.

\begin{lemma}\label{external}
Let $g \in \mathcal{G}_{1}^{2}$, $\eta\in(0,1)$ and $m\eta\geq1$. Then, for any $y\in\mathbb{R}^{d}$, $s=0,1,\cdots,N$, it follows
\begin{align}\label{exter2}
\left|\nabla V_g(s,y)\right|\leq e^{Cms\eta},
\end{align}
\begin{align}\label{exter3}
\left\|\nabla^{2}V_g(s,y)\right\|_{\rm HS}\leq C_{1}e^{C_{2}ms\eta},
\end{align}
where $C,C_{1},C_{2}$ are certain positive constants that do not depend on $m$ or $\eta$.
\end{lemma}

In addition, the following moment estimate of $\tilde{z}_{s}$ is also necessary and the proof is postponed to \ref{moment estimate}.
\begin{lemma}\label{exSGD}
Let $\tilde{z}_{s}=z_{sm}$, $\eta\in(0,1)$ and $m\eta\geq1$. Then, for any $s\geq0$, it follows
\begin{align}\label{exSGDforth}
\mathbb{E}|\tilde{z}_{s}|^{2}\leq C_{1}e^{C_{2}ms\eta}(1+\mathbb{E}|\tilde{z}_{0}|^{2}),
\end{align}
where $C_{1},C_{2}$ are certain positive constants that do not depend on $m$ or $\eta$.
\end{lemma}

\subsection{Proof of Theorem \ref{main}}

~~~~Combining with Proposition \ref{regularity}, we will use Lindeberg principle to state the main approximation theorem.

\begin{proof}
{ Following the Lindeberg principle, we employ the following iterative procedure. Given any $x\in\mathbb{R}^{d}$ and integers $0\leq r\leq s$, $\tilde{Y}_{s}(r,x)$ denotes the random variable $\tilde{Y}_{s}$ conditioned on $\tilde{Y}_{r}=x$, and $\tilde{z}_{s}(r,x)$ is defined analogously. Then we have
\begin{align*}
\mathbb{E}\left[g(\tilde{Y}_{s})\right]=\mathbb{E}\left[g(\tilde{Y}_{s}(1,\tilde{Y}_{1}))\right]-\mathbb{E}\left[g(\tilde{Y}_{s}(1,\tilde{z}_{1}))\right]
+\mathbb{E}\left[g(\tilde{Y}_{s}(1,\tilde{z}_{1}))\right].
\end{align*}
When $s>1$, since $\tilde{Y}_{s}=\tilde{Y}_{s}(r,\tilde{Y}_{r})$ for any $0\leq r\leq s$, we then have
\begin{align*}
\mathbb{E}\left[g(\tilde{Y}_{s}(1,\tilde{z}_{1}))\right]=
\mathbb{E}\left[g(\tilde{Y}_{s}(2,\tilde{Y}_{2}(1,\tilde{z}_{1})))\right]-\mathbb{E}\left[g(\tilde{Y}_{s}(2,\tilde{z}_{2}))\right]+\mathbb{E}\left[g(\tilde{Y}_{s}(2,\tilde{z}_{2}))\right].
\end{align*}
Continuing iteratively by applying the condition $\tilde{Y}_{s}=\tilde{Y}_{s}(r,\tilde{Y}_{r})$ for all integers $0\leq r\leq s$, it immediately follows
\begin{align*}
\mathbb{E}\left[g(\tilde{Y}_{s})\right]-\mathbb{E}\left[g(\tilde{z}_{s})\right]
=\sum_{i=1}^{s}\left[\mathbb{E}[g(\tilde{Y}_{s}(i,\tilde{Y}_{i}(i-1,\tilde{z}_{i-1})))]-\mathbb{E}[g(\tilde{Y}_{s}(i,\tilde{z}_{i}))]\right].
\end{align*}
Notice that $\tilde{Y}_{s}$ is a time homogeneous Markov Chain, then it follows $V_{g}(s-i,y)=\mathbb{E}\left[g(\tilde{Y}_{s})|\tilde{Y}_{i}=y\right]$. Hence, by the same argument as the proof of (\ref{Linde}), we obtain
\begin{align*}
\mathbb{E}\left[g(\tilde Y_s)\right]-\mathbb{E}\left[g(\tl z_{s})\right]=\sum_{i=1}^{s}\left[\mathbb{E}V_{g}\big(s-i,\tl Y^{\tl z_{i-1}}_{1}\big)-\mathbb{E}V_{g}\big(s-i,\tl z^{\tl z _{i-1}}_{1}\big)\right].
\end{align*}
}
Thus, notice that ${\frac{V_{g}(s,\cdot)}{C_{1}e^{C_{2}ms\eta}}\in\mathcal{G}_{1}^{2}}$, Proposition \ref{regularity} and Lemma \ref{exSGD} imply that
\begin{align*}
|\mathbb{E}g(\tilde Y_s)-\mathbb{E}g(\tl z_{s})|
\le&\sum_{i=1}^{s}\Big|\mathbb{E}V_{g}\big(s-i,\tl Y^{\tl z_{i-1}}_{1}\big)-\mathbb{E}V_{g}\big(s-i,\tl z^{\tl z_{i-1}}_{1}\big)\Big|\\ \le&C_{1}m\eta^{2}\sum_{i=1}^{s}e^{C_{2}m(s-i+1)\eta}\big(1+\mathbb{E}|\tilde{z}_{i-1}|^{2}\big)\\
\leq&C_{1}ms\eta^{2}e^{C_{2}ms\eta}\big(1+\mathbb{E}|\tilde{z}_{0}|^{2}\big)\leq C\eta,
\end{align*}
here the last inequality is obtained by $ms\eta\leq MN\eta\leq T$ and the desired result follows.
\end{proof}

\section{Numerical Experiments}
~~~~In the following, we conduct two numerical simulations to verify the effectiveness and practicality of the proposed model and algorithm.

\subsection{Two-dimensional Quadratic Optimization Problem}

~~~~Consider the following optimization problem with $n = 3, d = 2 $: for any $z=(z_{1},z_{2})\in\mathbb{R}^{2}$,
\begin{align*}
f_1(z) &= z_{(1)}^2, \ \ f_2(z) = z_{(2)}^2,\ \
f_3(z) = \delta \cos\left(z_{(1)}/\epsilon\right) \cos\left(z_{(2)}/\epsilon\right),
\epsilon = 0.1, \ \delta = 0.2,\\
z^* &= \arg \min_{x \in \mathbb{R}^2} F(z),\ \ \
F(z) := \frac{1}{3} \sum_{i=1}^3 f_i(z).
\end{align*}
Note that
\[
  \nabla f_1(z)=\begin{pmatrix}2z_{(1)}\\0\end{pmatrix},
  \quad
  \nabla f_2(z)=\begin{pmatrix}0\\2z_{(2)}\end{pmatrix},
  \quad
  \nabla f_3(z)=\begin{pmatrix}-2\sin(10z_{(1)})\cos(10z_{(2)})\\-2\cos(10z_{(1)})\sin(10z_{(2)})\end{pmatrix}.
\]
In this case,
$ f_1(z), f_2(z), f_3(z), F(z)$ are nonconvex functions.

We first apply the SVRG method as follows:
\begin{align*}
 z_{k}=z_{k-1}-\eta\Big[\nabla f_{\gamma_{k}}(z_{k-1})-\nabla f_{\gamma_{k}}(z_{\lceil\frac{k}{m}\rceil m})+\nabla F(z_{\lceil\frac{k}{m}\rceil m})\Big],\quad k\geq1,
\end{align*}
where $\gamma_k$ is randomly chosen from $\{1,2,3\}$ and independent of each other.

Next, we adopt the SDME method by the Euler-Maruyama scheme as follows:
\begin{align*}
\hat{Y}_{k}=\hat{Y}_{k-1}-\eta\nabla F(\hat{Y}_{k-1})+\eta\sigma(\hat{Y}_{k-1},{{\hat{Y}_{\lceil\frac{k}{m}\rceil m}}})^{\frac{1}{2}}\xi_{k},
\end{align*}
here $\xi_{k}$ denotes a sequence of $d$-dimensional independent and identically distributed random vectors, each following a standard normal distribution.

{Let $\eta=0.01,\ m=2n=6,\ x_{0}=\begin{pmatrix} 1 \\ 1.5 \end{pmatrix}$, where $\eta$ denotes the step size of the SVRG updates, $m$ is the inner loop length, simulation results for SVRG and SDME are shown in Figure~\ref{e1fig1}.
Here we quantify the approximation error between SVRG and SDME on the optimization problem by comparing their first and second order moments.
Specifically, the figure plots $|\mathbb{E}[z_k]|$ and $\bigl|\mathrm{Cov}(z_k)/\eta\bigr|^{0.5}$,
where $|\mathbb{E}[z_k]|$ captures the average convergence trajectory of the parameter vector,
while $\bigl|\mathrm{Cov}(z_k)/\eta\bigr|^{0.5}$ quantifies the magnitude of stochastic fluctuations.
For the SVRG algorithm, empirical means and covariances are computed by averaging over
100 independent runs. For the SDME model, sample trajectories are generated using the
Euler--Maruyama numerical formulas.
As shown in Figure~\ref{e1fig1}, the moments predicted by the SDME closely match the empirical
moments of SVRG, indicating that the continuous time SDME
model can accurately capture the behavior of the discrete SVRG iterations and thus provides an
effective approximation framework for theoretical analysis.}

\begin{figure}
    \centering
        \centering
        \includegraphics[scale=0.5]{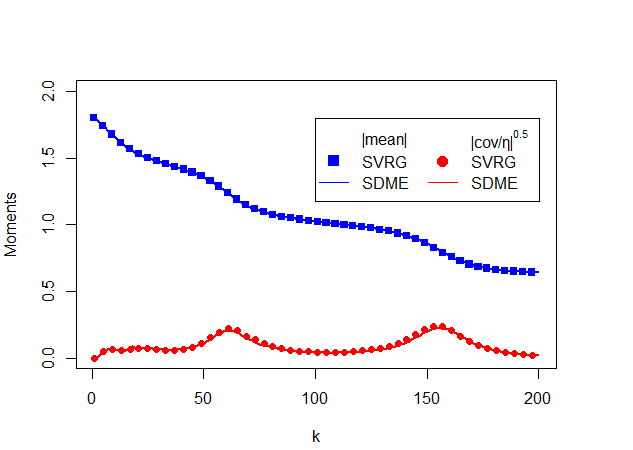}
    \caption{simulation results for SVRG and SDME}  
    \label{e1fig1}
\end{figure}    

\subsection{Binary Logistic Regression Optimization Problem on Real Data}

~~~~In this subsection, we use the Breast Cancer Wisconsin (Diagnostic) data set, see \url{https://archive.ics.uci.edu/dataset/17/breast+cancer+wisconsin+diagnostic}, to validate our theoretical approximation between SVRG and SDME.
This data set contains $n = 569$ samples obtained from digitized Fine Needle Aspirate images of breast masses. Each sample is described by 30 real valued features characterizing the morphology of cell nuclei, and is labeled as either benign or malignant.

We append a constant bias term to each feature vector, resulting in a total dimension of $d = 31$. We then discuss the $\ell_2$ regularized logistic regression with regularization parameter $\lambda = 0.01$, and evaluate both the SVRG algorithm and its SDME approximation under different hyperparameter settings to compare their performance on the following optimization problem:

\begin{align*}
f_i(z) =  \log\bigl(1 + \exp(-y_i\,x_i^\top z)\bigr) + \tfrac{\lambda}{2}\|z\|_2^2,\ \
\min_{z \in \mathbb{R}^{31}} F(z) = \frac{1}{n}\sum_{i=1}^n f_i(z)
\end{align*}
Note that
\begin{align*}
\nabla F(z)
= \frac{1}{n}\sum_{i=1}^n \frac{-y_i\,x_i}{1 + \exp(y_i\,x_i^\top z)} + \lambda\,z,
\end{align*}
where $x_i\in\mathbb{R}^{31}$ (30 standardized features plus a bias term) is the feature vector, $y_i\in\{-1,+1\}$ indicates benign or malignant labels, and $\lambda=0.01$ is the regularization parameter.

{As in Section~4.1, we apply the same SVRG updates and SDME formulas to the
$\ell_2$ regularized logistic regression experiment.
Let $\eta = 0.01$, $m = n = 569$, and $w_0 = 0$ (the zero vector).
Figure~\ref{e2fig1} represents the average trajectory and stochastic variability of the
SVRG algorithm across outer iterations
and shows that the SDME predictions continue to track the empirical SVRG results
closely even in high dimensional settings.}

\begin{figure}
    \centering
        \includegraphics[scale=0.5]{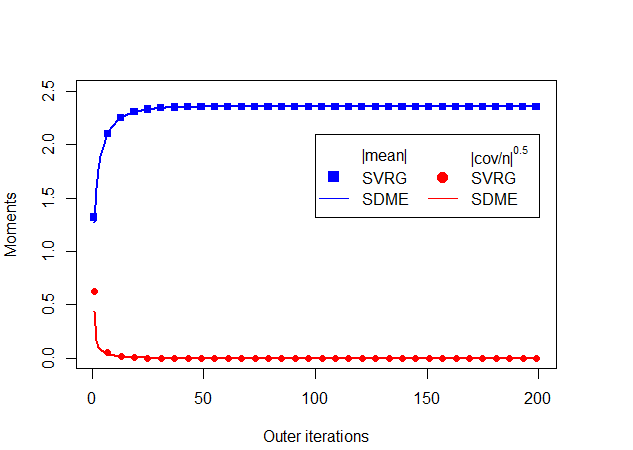}
    \caption{simulation results for SVRG and SDME}  
    \label{e2fig1}
\end{figure}    

\bigskip
\section{Conclusions}\label{conclusion}

~~~~In this article, we present a local block-wise approach to partition the SVRG algorithm and SDME mehtod, where each block constitutes a Markov process and the entire system forms a functional-valued Markov process. Subsequently, by employing the Lindeberg principle and Markov property, we compare the error between SVRG and the corresponding SDME, then obtain order 1 weak approximation in the smooth Wasserstein distance. The theoretical conclusions are further validated through numerical experiments.

Our approximation framework bridges SVRG with SDME theory. Drawing upon SDME analysis,
we demonstrate that including a delay in the diffusion term induces dissipation to mitigate fluctuations induced by Brownian motion.
This result provides an intuitive explanation for SVRG's intrinsic variance reduction mechanism.

Since obtaining exact solutions for degenerate SDME is generally intractable, particularly when
$d\geq 1$, existing studies model $Y_t$ in the form of an asymptotic expansion:
$ Y_t = Y_{0,t} + \sqrt{\eta}\,Y_{1,t} + \cdots,$
here $Y_{j,t}$ is a stochastic process satisfying the initial conditions $Y_{0,0}=y_{0}$,
with $Y_{j,0}=y_{0}$, $j\geq 1$ (see more details in \cite[Sections 3 and 4]{LTE17}).
Substituting this series into the stochastic modified equation (SME), we derive its asymptotic expansion,
which is then employed to approximate the SGD algorithm.
Future work aims to apply this framework to SDME and investigate the approximation of the SVRG algorithm by asymptotically expanded SDME.

\vspace{1cm}

\begin{appendix}

\section{Proof of Lemma \ref{mainlem1}}\label{Aind}

~~~~For brevity, set $b(y):=-\nabla F(y)$, $\Sigma_{z_{0}}(y)=\sigma(y,z_{0}).$ Where unambiguous, abbreviate $\Sigma_{z_{0}}(y)$ as $\Sigma(y).$ Therefore, $b(y), \Sigma(y)\in\mathcal{G}^{2}$ and recall (\ref{sfde-0-0}), the SDE can be reformulated as:
\begin{align}\label{SDE}
\dif Y_{t}=b(Y_{t})\dif t+\sqrt{\eta}\Sigma(Y_{t})\dif W_{t},\quad { Y_{0}}=y.
\end{align}
From \eqref{e:Tr=L2Norm}, it follows directly that equation \eqref{SDE} is degenerate, which renders the Malliavin calculus approach employed in \cite{CLX22} inapplicable. Here we use the Jacobian flow of the process $Y_{t}^{y}$ and the chain rule of the derivative to obtain the regularities of the semigroup $P_{t}g(y)=\mathbb{E}[g(Y_{t}^{y})]$ for any $t\in[0,m\eta]$ and $y\in\mathbb{R}^{d}$.

For a vector $u\in\mathbb{R}^{d},$ the Jacobian flow $\nabla_{u}Y_{t}^{y}$ in the direction of $u$ is given as
\begin{align*}
\nabla_{u}Y_{t}^{y}=\lim_{\epsilon\rightarrow0}\frac{Y_{t}^{y+\epsilon u}-Y_{t}^{y}}{\epsilon},\quad t\geq0.
\end{align*}
Under the above smooth conditions of the coefficients, the above Jacobian flow exists and satisfies the following differential equation:
\begin{align}\label{gradient SDE}
d\nabla_{u}Y_{t}^{y}=\nabla b(Y_{t}^{y})\nabla_{u}Y_{t}^{y}\dif t+\sqrt{\eta}\nabla\Sigma(Y_{t}^{y})\nabla_{u}Y_{t}^{y}\dif W_{t}, \quad \nabla_{u}Y_{0}^{y}=u.
\end{align}

Define $\nabla_{u_{2}}\nabla_{u_{1}}Y_{t}^{y}$ with $u_{1},u_{2}\in\mathbb{R}^{d},$ it follows
\begin{align}\label{secequation}
d\nabla_{u_{2}}\nabla_{u_{1}}Y_{t}^{y}=&\nabla b(Y_{t}^{y})\nabla_{u_{2}}\nabla_{u_{1}}Y_{t}^{y}\dif t+\nabla^{2}b(Y_{t}^{y})\nabla_{u_{2}}Y_{t}^{y}\nabla_{u_{1}}Y_{t}^{y}\dif t\nonumber\\
&+\sqrt{\eta}\nabla\Sigma(Y_{t}^{y})\nabla_{u_{2}}\nabla_{u_{1}}Y_{t}^{y}\dif W_{t}+\sqrt{\eta}\nabla^{2}\Sigma(Y_{t}^{y})\nabla_{u_{2}}Y_{t}^{y}\nabla_{u_{1}}Y_{t}^{y}\dif W_{t},
\end{align}
with $\nabla_{u_{2}}\nabla_{u_{1}}Y_{0}^{y}=0.$

Thus, the following estimate holds immediately.

\bigskip
\begin{lemma}
Let $\eta\in(0,1)$ and $t\in[0,m\eta]$. Then for any $y\in\mathbb{R}^{d}$ and $u,u_{1},u_{2}\in\mathbb{R}^{d}$, it follows that
\begin{align}\label{grad3}
\mathbb{E}|\nabla_{u}Y_{t}^{y}|^{4}\leq |u|^{4}e^{Cm\eta},
\end{align}
\begin{align}\label{gradsec}
\mathbb{E}|\nabla_{u_{2}}\nabla_{u_{1}}Y_{t}^{y}|^{2}
\leq C_{1}e^{C_{2}m\eta}|v_{1}|^{2}|v_{2}|^{2},
\end{align}
here the positive constants $C,C_{1},C_{2}$ do not depend on $m$ or $\eta$.
\end{lemma}

\begin{proof}
Recalling (\ref{gradient SDE}), combining with It\^{o}'s formula and (\ref{e:RelEigHS}), one can derive that
\begin{align*}
\frac{\dif}{\dif s}\mathbb{E}|\nabla_{u}Y_{s}^{y}|^{4}=&4\mathbb{E}[|\nabla_{u}Y_{s}^{y}|^{2}\langle\nabla b(Y_{s}^{y})\nabla_{u}Y_{s}^{y},\nabla_{u}Y_{s}^{y}\rangle]+2\eta\mathbb{E}[|\nabla_{u}Y_{s}^{y}|^{2}\|\nabla\Sigma(Y_{s}^{y})\nabla_{u}Y_{s}^{y}\|_{{\rm HS}}^{2}]\\
&+4\eta\mathbb{E}[|\nabla\Sigma(Y_{s}^{y})\nabla_{u}Y_{s}^{y}\nabla_{u}Y_{s}^{y}|^{2}]\\
\leq&C\mathbb{E}|\nabla_{v}Y_{s}^{x}|^{4}.
\end{align*}
The preceding inequality, along with the Gronwall inequality and the initial value $\nabla_{u}Y_{0}^{y}=u$, yields
\begin{align*}
\mathbb{E}|\nabla_{u}Y_{t}^{y}|^{4}\leq |u|^{4}e^{Ct}\leq |u|^{4}e^{Cm\eta}.
\end{align*}

Denote $\phi(t)=\nabla_{u_{2}}\nabla_{u_{1}}Y_{t}^{y}$, according to the equation \eqref{secequation}, by It\^{o}'s formula, the Cauchy-Schwarz inequality and Young's inequality, one can obtain that
\begin{align*}
\frac{\dif}{\dif s}\mathbb{E}|\phi(s)|^{2}=&2\mathbb{E}\big[\langle\nabla b(Y_{s}^{y})\phi(s)+\nabla^{2}b(Y_{s}^{y})\nabla_{u_{2}}Y_{s}^{y}\nabla_{u_{1}}Y_{s}^{y},\phi(s)\rangle\big]\\
&+\eta\mathbb{E}\big[\|\nabla\Sigma(Y_{s}^{y})\phi(s)+\nabla^{2}\Sigma(Y_{s}^{y})\nabla_{u_{2}}Y_{s}^{y}\nabla_{u_{1}}Y_{s}^{y}\|_{{\rm HS}}^{2}\big]\\
\leq&C\left(\mathbb{E}|\phi(s)|^{2}+\mathbb{E}[|\nabla_{u_{2}}Y_{s}^{y}||\nabla_{u_{1}}Y_{s}^{y}||\phi(s)|]+\mathbb{E}\big[|\nabla_{u_{2}}Y_{s}^{y}|^{2}|\nabla_{u_{1}}Y_{s}^{y}|^{2}\big]\right)\\
\leq&C\left(\mathbb{E}|\phi(s)|^{2}+\mathbb{E}[|\nabla_{u_{2}}Y_{s}^{y}|^{2}|\nabla_{u_{1}}Y_{s}^{y}|^{2}]\right),
\end{align*}
then the Cauchy-Schwarz inequality and (\ref{grad3}) further implies that
\begin{align*}
\frac{\dif}{\dif s}\mathbb{E}|\phi(s)|^{2}\leq&C\left(\mathbb{E}|\phi(s)|^{2}+\mathbb{E}[|\nabla_{u_{2}}Y_{s}^{y}|^{4}|\nabla_{u_{1}}Y_{s}^{y}|^{4}]\right)
\leq C_{1}\left(\mathbb{E}|\phi(s)|^{2}+e^{Cm\eta}|u_{1}|^{2}|u_{2}|^{2}\right),
\end{align*}
where the constants $C,C_{1}>0$ are independent of $m$ and $\eta$. This inequality, Gronwall's inequality and the initial value $\phi(0)=0$, together imply
\begin{align*}
\mathbb{E}|\phi(t)|^{2}\leq C_{1}e^{Cm\eta}|u_{1}|^{2}|u_{2}|^{2}\left(t+\int_{0}^{t}se^{C_{1}(t-s)}\dif s\right)
\leq C_{1}e^{C_{2}m\eta}|v_{1}|^{2}|v_{2}|^{2},
\end{align*}
here the positive constants $C_{1},C_{2}$ are independent of $m$ and $\eta$. The proof is complete.
\end{proof}
\bigskip

Now, we are in the position to prove Lemma \ref{mainlem1}.

\begin{proof}[Proof of Lemma \ref{mainlem1}]
Notice that for any $g\in \mathcal{G}^{2}_{1}$, the semigroup $P_{t}g(y)=\mathbb{E}[g(Y_{t}^{y})]$, by Lebesgue's dominated convergence theorem, the chain rule, the Cauchy-Schwarz inequality and (\ref{grad3}), we have
\begin{align*}
|\nabla P_{t}g(y)|\leq\sup_{|u|=1}|\mathbb{E}[\nabla g(Y_{t}^{y})\nabla_{u}Y_{t}^{y}]|
\leq\sup_{|u|=1}\|\nabla g\|\mathbb{E}|\nabla_{u}Y_{t}^{y}|\leq e^{Cm\eta},
\end{align*}
\eqref{gradient2} is proved. Combining with the Cauchy-Schwarz inequality, \eqref{grad3} and \eqref{gradsec}, one can obtain
\begin{align*}
\|\nabla^{2}P_{t}g(y)\|_{{\rm HS}}\leq&\sup_{|u_{1}|,|u_{2}|=1}\left(\left|\mathbb{E}\big[\nabla^{2}g(Y_{t}^{y})\nabla_{u_{2}}Y_{t}^{y}
\nabla_{u_{1}}Y_{t}^{y}\big]\right|+\left|\mathbb{E}\big[\nabla g(Y_{t}^{y})\nabla_{u_{2}}\nabla_{u_{1}}Y_{t}^{y}\big]\right|\right)\\
\leq&\sup_{|u_{1}|,|u_{2}|=1}\left(\mathbb{E}\big[\left|\nabla_{u_{2}}Y_{t}^{y}\right|
\left|\nabla_{u_{1}}Y_{t}^{y}\right|\big]+\mathbb{E}\big|\nabla_{u_{2}}\nabla_{u_{1}}Y_{t}^{y}\big|\right)\\
\leq&C_{1}e^{C_{2}m\eta},
\end{align*}
here the positive constants $C_{1}$ and $C_{2}$ do not depend on $m$ or $\eta$. The second inequality is proved and the proof is complete.
\end{proof}

\section{Proof of Lemmas in Section \ref{forth}}\label{moment estimate}

\begin{proof}[Proof of Lemma \ref{fourth}]
According to the SDE (\ref{sfde-0-0}), invoking the It\^{o} formula, (\ref{e:Tr=L2Norm}) and Young's inequality, one can obtain
\begin{eqnarray*}
\frac{\dif}{\dif s}\E |Y^{y}_{s}|^{2}&=&2 \E \Ll Y_{s}, -\nabla F(Y_{s})\Rr +\eta\E\|\sigma(Y_{s}, z _{0})^{\frac{1}{2}}\|^{2}_{\rm HS}\\
& \le & C\left(\E |Y^{y}_{s}|^{2}+\E|X_s||\nabla F(0)|+\eta|Y_{s}- z _{0}|^{2}\right)\\
& \le &C\left(\E |Y^{y}_{s}|^{2}+|\nabla F(0)|^2+\eta\mathbb{E}| z _{0}|^{2}\right).
\end{eqnarray*}
Noting that $Y^{y}_{0}=y,$ one can write by the Gronwall inequality that
\begin{eqnarray*}
\E |Y^{y}_{t}|^{2}
&\leq&e^{Ct}|y|^{2}+C(|\nabla F(0)|^2+\eta\mathbb{E}|z_{0}|^{2})\int_{0}^{t}e^{C(t-s)}\dif s\\
&\leq&C_{1}(1+|y|^{2}+\mathbb{E}|z_{0}|^{2})e^{C_{2}t}\leq C_{1}(1+|y|^{2}+\mathbb{E}|z_{0}|^{2})e^{C_{2}m\eta},
\end{eqnarray*}
here $C_{1},C_{2}$ both are independent of $m$ and $\eta$. Then \eqref{moment} is proved.

Invoking It\^{o}'s isometry, the Cauchy-Schwarz inequality and (\ref{e:Tr=L2Norm}) yields
\begin{eqnarray*}
\mathbb{E}|Y_{t}^{y}-y|^{2}&\leq&2\mathbb{E}|\int_{0}^{t}-\nabla F(Y_{r}^y)\dif r|^{2}+2\mathbb{E}|\int_{0}^{t}\left(\eta\sigma(Y_{r},z_0)\right)^{\frac{1}{2}}\dif W_{r}|^{2}\nonumber\\
&\leq&2t\int_{0}^{t}\mathbb{E}|\nabla F(Y_{r}^{y})|^{2}\dif r+2\eta\int_{0}^{t}\mathbb{E}\|\sigma(Y_{r},z_0)^{\frac{1}{2}}\|^{2}_{{\rm {HS}}}\dif r\nonumber\\
&\leq&C\left(t\int_{0}^{t}\big(|\nabla F(0)|^{2}+\mathbb{E}|Y_{r}|^{2}\big)\dif r+\eta\int_{0}^{t}\mathbb{E}|Y_{s}-z_{0}|^{2}\dif r\right)\\
&\leq&C\left((t+\eta)\int_{0}^{t}\mathbb{E}|Y_{r}|^{2}\dif r+t\big(t+\eta\mathbb{E}|z_{0}|^{2}\big)\right),
\end{eqnarray*}
combing with \eqref{moment}, \eqref{moment-1} holds immediately.
\end{proof}

\begin{proof}[Proof of Lemma \ref{coupling2}]
\bigskip
Considering (\ref{representation}), notice that for $i=1,\cdots,n$, $ f_{i}(y)$ is $M$-smooth, then it is easy to derive from Young's inequality that
\begin{align*}
\mathbb{E}| z _{k}|^{2}=
&\mathbb{E}| z _{k-1}|^{2}-2\eta\mathbb{E}\left\langle z _{k-1},\nabla f_{\gamma_{k}}( z _{k-1})-\nabla f_{\gamma_{k}}( z _{0})+\nabla F( z _{0})\right\rangle\\
&+\eta^{2}\mathbb{E}\left|\nabla f_{\gamma_{k}}( z _{k-1})-\nabla f_{\gamma_{k}}( z _{0})+\nabla F( z _{0})\right|^{2}\\
\leq&(1+\eta)\mathbb{E}| z _{k-1}|^{2}+\eta(1+\eta)\mathbb{E}\left|\nabla f_{\gamma_{k}}( z _{k-1})-\nabla f_{\gamma_{k}}( z _{0})+\nabla F( z _{0})\right|^{2}\\
\leq&(1+C\eta)\mathbb{E}| z _{k-1}|^{2}+C\eta(1+\mathbb{E}| z _{0}|^{2}),
\end{align*}
here $C>0$, and is independent of $m$ and $\eta$. Therefore, by the inequality $1+r\leq e^{r}$ for any $r>0$, we have
\begin{align*}
\mathbb{E}| z _{k}^{y}|^{2}
\le & (1+C\eta)^{k}|y|^{2}+C(1+\mathbb{E}|z_{0}|^{2})\eta\sum_{j=0}^{k-1}
(1+C\eta)^{j} \\
\leq&C_{1}(1+|y|^{2}+\mathbb{E}|z_{0}|^{2})(1+C\eta)^{k}\\
\leq&C_{1}(1+|y|^{2}+\mathbb{E}|z_{0}|^{2}))e^{C_{2}\eta k}\\
\leq&C_{1}(1+|y|^{2}+\mathbb{E}|z_{0}|^{2}))e^{C_{2}m\eta},
\end{align*}
here positive constants $C_{1},C_{2}$ are independent of $m$ and $\eta$. The proof is complete.
\end{proof}

\begin{proof}[Proof of Lemma \ref{SGDtaylor}]
For $Y_{\eta}^{y}$, one can write by It$\hat{o}$'s formula that
\begin{align*}
\mathbb{E}v_{t}(Y_{\eta}^{y})-v_{t}(y)=\int_{0}^{\eta}\mathbb{E}\left[\mathcal{A}v_{t}(Y_{s}^{y})\right]\dif s,
\end{align*}
where $\mathcal{A}$ is the infinitesimal generator of the Markov process $Y_{t}$, which is defined as follows:
\begin{align*}
\mathcal{A}g(y)
=&\frac{1}2\eta\Ll \sigma(y,z_0),\nabla^{2} g(y)\Rr_{{\rm HS}}-\Ll \nabla F(y),\nabla g(y)\Rr, \quad g\in\mathcal{C}^{2}(\mathbb{R}^{d}).
\end{align*}
For $z_{1}^{y}$, notice that $z_{1}^{y}=y-\eta\nabla F(y)+\sqrt{\eta}Q_{\eta}(y,z_{0},\gamma)$, by the second order Taylor expansion, we have
\begin{align*}
v_{t}(z_{1}^{y})-v_{t}(y)
=&\langle-\eta\nabla F(x)+\sqrt{\eta}Q_{\eta}(y,z_{0},\gamma),\nabla v_{t}(y)\rangle\\
&+\int_{0}^{1}\int_{0}^{r}\left\langle\left(-\eta\nabla F(y)+\sqrt{\eta}Q_{\eta}(y,z_{0},\gamma)\right)\left(-\eta\nabla F(x)+\sqrt{\eta}Q_{\eta}(y,z_{0},\gamma)\right)^{*},\right.\\
&\left.\qquad\qquad\qquad\nabla^{2}v_{t}\left(y+s(-\eta \nabla f_\gamma(y)+\eta \nabla f_\gamma(z_0)-\eta\nabla F(z_0))\right)\right\rangle_{{\rm HS}}\dif s\dif r,
\end{align*}
then \eqref{e:ConV} further implies
\begin{align*}
\mathbb{E}v_{t}(z_{1}^{y})-v_{t}(y)
=&\langle-\eta\nabla F(y),\nabla v_{t}(y)\rangle\\
&+\int_{0}^{1}\int_{0}^{r}\mathbb{E}\left\langle\left(-\eta\nabla F(y)+\sqrt{\eta}Q_{\eta}(y,z_{0},\gamma)\right)\left(-\eta\nabla F(y)+\sqrt{\eta}Q_{\eta}(y,z_{0},\gamma)\right)^{*},\right.\\
&\left.\qquad\qquad\qquad\nabla^{2}v_{t}\left(y+s(-\eta \nabla f_\gamma(y)+\eta \nabla f_\gamma(z_0)-\eta\nabla F(z_0))\right)\right\rangle_{{\rm HS}}\dif s\dif r.
\end{align*}
Hence, we have
\begin{align*}
\mathbb{E}v_{t}(Y_{\eta}^{y})-\mathbb{E}v_{t}(z_{1}^{y})=\int_{0}^{\eta}\mathbb{E}\left[\left\langle\nabla F(y),\nabla v_{t}(y)\right\rangle-\left\langle\nabla F(Y_{s}^{y}),\nabla v_{t}(Y_{s}^{y})\right\rangle\right]\dif s+\mathcal{R}_{\eta}(y,z_{0})
\end{align*}
with
\begin{align*}
\mathcal{R}_{\eta}(y,z_{0})=&\frac{1}2\eta\int_{0}^{\eta}\mathbb{E}\Ll \sigma(Y_{s}^{y},z_0), \nabla^{2} v_{t}(Y_{s}^{y})\Rr_{{\rm HS}}\dif s\\
&+\int_{0}^{1}\int_{0}^{r}\mathbb{E}\left\langle\left(-\eta\nabla F(y)+\sqrt{\eta}Q_{\eta}(y,z_{0},\gamma)\right)\left(-\eta\nabla F(y)+\sqrt{\eta}Q_{\eta}(y,z_{0},\gamma)\right)^{*},\right.\\
&\left.\qquad\qquad\qquad\nabla^{2}v_{t}\left(x+s(-\eta \nabla f_\gamma(y)+\eta \nabla f_\gamma(z_0)-\eta\nabla F(z_0))\right)\right\rangle_{{\rm HS}}\dif s\dif r.
\end{align*}
On the one hand, notice that $\nabla F\in\mathcal{G}^{2}$, one can write by Lemma \ref{mainlem1}, the Cauchy-Schwarz inequality and \eqref{moment-1} that
\begin{align*}
&\left|\int_{0}^{\eta}\mathbb{E}\left[\left\langle\nabla F(y),\nabla v_{t}(y)\right\rangle-\left\langle\nabla F(Y_{s}),\nabla v_{t}(Y_{s}^{y})\right\rangle\right]\dif s\right|\\
\leq&\int_{0}^{\eta}\mathbb{E}\left[\left|\left\langle\nabla F(y),\nabla v_{t}(y)-\nabla v_{t}(Y_{s}^{y})\right\rangle\right|+\left|\left\langle\nabla F(y)-\nabla F(Y_{s}^{y}),\nabla v_{t}(Y_{s}^{y})\right\rangle\right|\right]\dif s\\
\leq& C(\|\nabla v_{t}\|+\|\nabla^{2}v_{t}\|_{{\rm HS}})(1+|y|)\int_{0}^{\eta}\mathbb{E}|Y_{s}^{y}-y|\dif s\\
\leq& C_{1}e^{C_{2}m\eta}(1+|y|^{2}+\mathbb{E}|z_{0}|^{2})\int_{0}^{\eta}\sqrt{s(s+\eta)}\dif s\leq C_{1}e^{C_{2}m\eta}(1+|y|^{2}+\mathbb{E}|z_{0}|^{2})\eta^{2},
\end{align*}
for some positive constants $C_{1},C_{2}$ that are independent of $m$ and $\eta$. On the other hand, since $\nabla P(\cdot),\sigma(\cdot,y)\in\mathcal{G}^{2}$, by \eqref{e:Tr=L2Norm} and Lemma \ref{mainlem1}, we obtain
\begin{align*}
\left|\mathcal{R}_{\eta}(y,z_{0})\right|
\leq& C\eta\|\nabla^{2}v_{t}\|\left(\int_{0}^{\eta}(1+\mathbb{E}|Y_{s}^{y}|^{2}+\mathbb{E}|z_{0}|^{2})\dif s+\eta\int_{0}^{1}\int_{0}^{r}(1+\mathbb{E}|Y_{s}^{y}|^{2}+\mathbb{E}|z_{0}|^{2})\dif s\dif r\right)\\
\leq&C_{1}e^{C_{2}m\eta}(1+|y|^{2}+\mathbb{E}|z_{0}|^{2})\eta^{2},
\end{align*}
for certain positive constants $C_{1},C_{2}$ that are independent of $m$ and $\eta$, the desired result follows.
\end{proof}

\begin{proof}[Proof of Lemma \ref{external}]

When $s=0$, we have $V_g(0,y)=\E \left[g(\tl Y^y_0)\right]=g(y)$, notice that $g\in\mathcal{G}_{1}^{2}$, \eqref{exter2} holds obviously. When $s=1$, \eqref{exter2} follows from \eqref{gradient2}.

Now, assume \eqref{exter2} holds for $s$, that is,
\begin{align*}
\left|\nabla V_g(s,y)\right|\leq e^{Cms\eta}
\end{align*}
for some positive constant $C$, which does not depend on $m$ or $\eta$. By the Markov property of $\tl Y^y_{s}$ and the chain rule, we have
\begin{align}\label{form1}
\nabla V_g(s+1,y)=\E \left[\nabla g(\tl Y^y_{s+1})\right]=\mathbb{E}\left[\nabla g(\tl Y^{\tl Y_{1}^{y}}_{s+1})\right]=\mathbb{E}\left[\nabla_{(\tl Y_{1}^{y})} g(\tl Y^{\tl Y_{1}^{y}}_{s+1})\nabla_{(y)}\tl Y_{1}^{y}\right],
\end{align}
where $\nabla_{(y)}$ means the gradient operator acts on $y$. These implies that
\begin{align*}
\left|\nabla V_g(s+1,y)\right|\leq&\mathbb{E}\left|\nabla_{(\tl Y_{1}^{y})} g(\tl Y^{\tl Y_{1}^{y}}_{s+1})\nabla_{(y)}\tl Y_{1}^{y}\right|\\
\leq&\mathbb{E}\left[\mathbb{E}\left[\left|\nabla_{(\tl Y_{1}^{y})} g(\tl Y^{\tl Y_{1}^{y}}_{s+1})\right|\big|\tl Y_{1}^{y}\right]\left|\nabla_{(y)}\tl Y_{1}^{y}\right|\right]\\
\leq&e^{Cms\eta}\mathbb{E}\left|\nabla_{(y)}\tl Y_{1}^{y}\right|\\
\leq&e^{Cms\eta}\sup_{|u|=1}\mathbb{E}\left|\nabla_{u}\tl Y_{1}^{y}\right|\leq e^{Cm(s+1)\eta},
\end{align*}
here the last inequality is due to \eqref{grad3} and the Cauchy-Schwarz inequality. The inequality \eqref{exter2} is derived.

For the inequality \eqref{exter3}, when $s=0$, \eqref{exter3} follows from the assumption $g\in\mathcal{G}_{1}^{2}$. When $s\geq1$, we first use the induction method to prove the following inequality:
\begin{align}\label{exter4}
\left\|\nabla^{2}V_g(s,y)\right\|_{\rm HS}\leq C_{1}se^{C_{2}ms\eta}.
\end{align}
for certain positive constants $C_{1},C_{2}$ that are independent of $m$ and $\eta$. When $s=1$, \eqref{exter4} follows from \eqref{sec} immediately.

Now, we assume \eqref{exter4} holds for $s$. Then \eqref{form1} implies that
\begin{align*}
\nabla^{2}V_g(s+1,y)=\mathbb{E}\left[\nabla_{(\tl Y_{1}^{y})}^{2} g(\tl Y^{\tl Y_{1}^{y}}_{s+1})\nabla_{(y)}\tl Y_{1}^{y}\nabla_{(y)}\tl Y_{1}^{y}\right]+\mathbb{E}\left[\nabla_{(\tl Y_{1}^{y})} g(\tl Y^{\tl Y_{1}^{y}}_{s+1})\nabla^{2}_{(y)}\tl Y_{1}^{y}\right],
\end{align*}
where $\nabla^{2}_{(y)}$ means the operator acts on $y$. Thus we have
\begin{align*}
\left\|\nabla^{2}V_g(s+1,y)\right\|_{\rm HS}\leq \mathbb{E}\left|\nabla_{(\tl Y_{1}^{y})}^{2} g(\tl Y^{\tl Y_{1}^{y}}_{s+1})\nabla_{(y)}\tl Y_{1}^{y}\nabla_{(y)}\tl Y_{1}^{y}\right|+\mathbb{E}\left|\nabla_{(\tl Y_{1}^{y})} g(\tl Y^{\tl Y_{1}^{y}}_{s+1})\nabla^{2}_{(y)}\tl Y_{1}^{y}\right|.
\end{align*}
Then we can write by the Cauchy-Schwarz inequality and \eqref{grad3} that
\begin{align*}
\mathbb{E}\left|\nabla_{(\tl Y_{1}^{y})}^{2} g(\tl Y^{\tl Y_{1}^{y}}_{s+1})\nabla_{(y)}\tl Y_{1}^{y}\nabla_{(y)}\tl Y_{1}^{y}\right|
\leq&C_{1}se^{C_{2}ms \eta}\sup_{|u|=1}\mathbb{E}|\nabla_{u}\tl Y_{1}^{y}|^{2}\leq C_{1}se^{C_{2}(s+1)m\eta}
\end{align*}
whereas by \eqref{exter2} and \eqref{gradsec}
\begin{align*}
\mathbb{E}\left|\nabla_{(\tl Y_{1}^{y})} g(\tl Y^{\tl Y_{1}^{y}}_{s+1})\nabla^{2}_{(y)}\tl Y_{1}^{y}\right|
\leq&e^{C_{2}ms\eta}\sup_{|u_{1}|,|u_{2}|=1}\mathbb{E}\left|\nabla_{u_{1}}\nabla_{u_{2}}\tl Y_{1}^{y}\right|\leq C_{1}e^{C_{2}(s+1)m\eta}.
\end{align*}
These inequalities imply that
\begin{align*}
\left\|\nabla^{2}V_g(s+1,y)\right\|_{\rm HS}\leq C_{1}(s+1)e^{C_{2}(s+1)m\eta},
\end{align*}
\eqref{exter4} is proved. Then notice that $m\eta\geq1$, \eqref{exter3} follows from the inequality $r\leq e^{r}$ for any $r\geq0$.
\end{proof}

\begin{proof}[Proof of Lemma \ref{exSGD}]

With Lemma \ref{coupling2}, one can derive that
\begin{align*}
\mathbb{E}|\tilde{ z }_{1}|^{2}\leq C_{1}e^{C_{2}m\eta}\mathbb{E}|\tilde{ z }_{0}|^{2}+C_{1}e^{C_{2}m\eta}
\end{align*}
for some $C_{1},C_{2}>0$, which are independent of $m$ and $\eta$. Inductively, one can write by the inequality $r\leq e^{r}$ for any $r>0$ and $m\eta>1$ that
\begin{align*}
\mathbb{E}|\tilde{ z }_{s}|^{2}\leq& C_{1}^{s}e^{C_{2}ms\eta}\mathbb{E}|\tilde{ z }_{0}|^{2}+C_{1}e^{C_{2}m\eta}\sum_{j=1}^{s}C_{1}^{j}e^{C_{2}ms\eta}\\
\leq&C_{1}^{s}e^{C_{2}ms\eta}\mathbb{E}|\tilde{ z }_{0}|^{2}+C_{1}^{s+1}e^{C_{2}m(s+1)\eta}\\
\leq&C_{3}e^{C_{4}ms\eta}(1+\mathbb{E}|\tilde{ z }_{0}|^{2})
\end{align*}
for certain constants $C_{3},C_{4}>0$, which are independent of $m$ and $\eta$. The proof is complete.
\end{proof}
\end{appendix}

\section*{Acknowledgements}
We are grateful to the Editor and two anonymous referees for
their insightful comments and suggestions on this article, which have led to significant improvements.
Ting Zhang's research was supported by the National Natural
Science Foundation of China (12301342, W2521122).


\bibliographystyle{amsplain}

\end{document}